\documentclass[12pt]{amsart}
\usepackage{amssymb,amsmath,amsthm}
\usepackage{enumerate}
\usepackage{mathrsfs}
\newtheorem{theorem}{Theorem}[section]
\newtheorem{proposition}[theorem]{Proposition}
\newtheorem{lemma}[theorem]{Lemma}
\newtheorem{corollary}[theorem]{Corollary}
\theoremstyle{definition}
\newtheorem{definition}[theorem]{Definition}
\newtheorem{example}[theorem]{Example}
\newtheorem{remark}[theorem]{Remark}

\numberwithin{equation}{section}
	\title[ $\ast$-paranormal absolutely norm attaining]{Representation and normality of $\ast$-paranormal absolutely norm attaining  operators}
\author{Neeru Bala}

\address{Statistics Mathematics Unit, ISI Banaglore, RV College Mysore Road, Bangalore - 600019 }

\email{neerusingh@gmail.com}

\subjclass[2010]{47A10, 47B20}

\keywords{$\ast$-paranormal operator, Weyl's spectrum, essential spectrum, invariant subspace, compact operator, $\mathcal{AN}$-operators, isometry.}

\begin{document}

	\begin{abstract}
	In this article, we give a representation of $\ast$-paranormal absolutely norm attaining  operator. Explicitly saying, every $\ast$-paranormal absolutely norm attaining ($\mathcal{AN}$ in short)  $T$ can be decomposed as $U\oplus D$, where $U$ is a direct sum of scalar multiple of unitary operators and $D$ is a $2\times 2$ upper diagonal operator matrix. By the representation it is clear that the class of $\ast$-paranormal $\mathcal{AN}$-operators is bigger than the class of normal $\mathcal{AN}$-operators but here we observe that a $\ast$-paranormal $\mathcal{AN}$-operator is normal if either it is invertible or dimension of its null space is same as dimension of null space of its adjoint.
	\end{abstract}
	
	\maketitle
	
	\section{Introduction}
	One of the most promising areas of research in operator theory deals with the study of non-normal operators. Hyponormal and paranormal operators are classes of non-normal operators which received considerable attention in the literature, \cite{ANDOh,ANDO,FURUTA,ISTRA} are a few references.
	
	Another interesting class of non-normal operators is the class of $\ast$-paranormal operators. We say a bounded linear operator $T$ on a Hilbert space $H$ is $\ast$-paranormal, if 
	\begin{equation*}
		\|{T^*}^2x\|\leq\|Tx\|^2\|x\|,\; \text{ for all}\; x\in H.
	\end{equation*}
	It is proved in \cite{ARORA} that every hyponormal operator is $\ast$-paranormal. K. Tanahashi and A. Uchiyama \cite{TANA} studied the spectral properties of $\ast$-paranormal operators. It is known that the inverse of an invertible paranormal operator is paranormal \cite{ISTRA} but in \cite{TANA} K. Tanahashi and A. Uchiyama gave an example of invertible  $\ast$-paranormal operator whose inverse is not $\ast$-paranormal. This shows that the definition of paranormal and $\ast$-paranormal operators looks similar but they behave quite differently.
	
	For non-normal operators, one of the questions which is addressed by many researchers is that what are the properties which distinguish it from normal operators or in other words, under what assumptions the non-normal operator becomes normal. For paranormal and hyponormal operators, a few properties are given in \cite{FURUTA,FURUTA HN,ISTRA,PUTNAM}. Here we explore this question for $\ast$-paranormal operators. In \cite{RASHID}, it is proved that a compact $\ast$-paranormal operator is normal. We want to replace the space of all compact operators by a bigger class of $\mathcal{AN}$-operators and see what extra assumptions we require to get normality. In \cite{HYPO}, we gave a representation of hyponormal $\mathcal{AN}$-operator. In this article, we prove a similar representation for $\ast$-paranormal operator. We illustrated with an example that the representation looks similar but it need not belong to the same class.

	Absolutely norm attaining operators is motivated by the study of norm attaining operators. Norm attaining operators on Banach spaces is an extensively  studied class of operators from many years, see \cite{ACOSTA,BISHOP,LEE} for more details. The class of absolutely norm attaining operators ($\mathcal{AN}$-operators, in short) is a subclass of norm attaining operators, which is first studied by X. Carvajal and W. Neves \cite{CARVAJAL}. We say that a bounded linear operator $T$ on a Hilbert space $H$ is absolutely norm attaining, if $T|_{M}$ is norm attaining for every closed subspace $M\subseteq H.$

	The article is divided into three sections. In section two, we recall a few definitions and results of $\ast$-paranormal operators and $\mathcal{AN}$-operators. In section three, we prove our main result of a representation of $\ast$-paranormal $\mathcal{AN}$-operators.
	\section{Preliminaries}
	Throughout this article, we work with infinite dimensional Hilbert spaces, which we denote by $H,H_1,H_2$ etc.
	The space of all bounded linear operators from a Hilbert space $H_1$ to $H_2$ is denoted by $\mathcal{B}(H_1,H_2)$. We write $\mathcal{B}(H):=\mathcal{B}(H,H)$. If $T\in\mathcal{B}(H_1,H_2)$, then the \textit{null space} and \textit{range space }of $T$ are denoted by $N(T)$ and $R(T)$, repectively.
	
	Let $T\in\mathcal{B}(H)$ and $M$ be a closed subspace of $H$. Then $M$ is said to be an \textit{invariant} subspace for $T$, if $TM\subseteq M$, and is said to be a \textit{reducing} subspace for $T$, if both $M$ and $M^{\perp}$ are invariant subspaces for $T$.
	
 %The operator $T$ is said to be \textit{Fredholm}, if $R(T)$ is closed, $N(T)$ and $N(T^*)$ are finite-dimensional. The set
%	$$\sigma_{\text{ess}} (T):=\{\lambda\in\mathbb{C}:T-\lambda I\text{ is not Fredholm}\},$$
%	is called the \textit{essential spectrum }of $T$. For more details about Fredholm theory and essential specrtum, we refer to \cite{MULLER}.
	
	An operator $T\in\mathcal{B}(H)$ is called \textit{normal}, if $TT^*=T^*T$ and \textit{self-adjoint}, if $T=T^*$. Further, $T$ is said to be \textit{positive}, if $T$ is self-adjoint and $\langle Tx,x\rangle\geq 0$ for all $x\in H$. For a positive operator $T$, there exist a unique positive operator $S$ satisfying $T=S^2$. The operator $S$ is called the \textit{square root} of $T$ and is denoted by $T^{1/2}$. In this article, we deal with the following classes of non-normal operators.
	
		For $T\in\mathcal{B}(H)$, the \textit{resolvent} of $T$ is the set $\rho(T):=\{\lambda\in\mathbb{C}:T-\lambda I \text{ is invertible in }\mathcal{B}(H)\}$, and the \textit{spectrum} of $T$ is $\sigma(T):=\mathbb{C}\setminus\rho(T)$. If $T=T^*$, then the \textit{discrete spectrum} $\sigma_d(T)$ of $T$ is determined by
	$$\sigma_d(T)=\{\lambda\in\sigma(T):\lambda\text{ is an isolated eigenvalue with finite multiplicity}\}.$$
	The set $\sigma(T)\setminus\sigma_d(T)$ is called the essential spectrum of $T$. If $T$ is not self-adjoint, then a more general definition of essential spectrum is given in \cite{MULLER}. For more details, we refer to \cite{REED, MULLER}.
	
	Throughout the article, we use the properties of the essential spectrum. We use the following description for the essential spectrum.
	\begin{theorem}\cite[Theorem 7.11, Page 236]{REED}
		If $T=T^*$, then $\lambda\in\sigma_{\text{ess}}(T)$ if and only if one of the following holds:
		\begin{enumerate}
			\item $\lambda$ is an eigenvalue of infinite multiplicity.
			\item  $\lambda$ is a limit point of the set of eigenvalues of $T$.
			\item $\lambda\in\sigma_c(T)=\{\delta\in\sigma(T):N(T-\delta I)=\{0\},\,\overline{R(T-\delta I)}=H,\, R(T-\delta I)\text{ is not closed}\}$.
		\end{enumerate}
	\end{theorem}
	
	\begin{definition}
		Let $T\in\mathcal{B}(H)$. Then $T$ is said to be
		\begin{enumerate}
			\item  \textit{hyponormal}, if $\|T^*x\|\leq\|Tx\|$ for every $x\in H$.
			\item  \textit{paranormal}, if $\|Tx\|^2\leq\|T^2x\|\|x\|$ for every $x\in H$.
			\item  $\ast$-\textit{paranormal}, if $\|T^*x\|^2\leq\|T^2x\|\|x\|$ for every $x\in H$.
		\end{enumerate}
	\end{definition}
	It is easy to observe that every hyponormal operator is paranormal as well as $\ast$-paranormal. Note that $T$ is hyponormal if and only if $TT^*\leq T^*T.$ For more details about these classes of operators, we refer to \cite{ARORA,FURUTA,FURUTA HN,HAN,ISTRA,MARTIN,TANA}. The following result gives a characterization for $\ast$-paranormal operators.
	\begin{theorem}\cite{ARORA,HAN}
		Let $T\in\mathcal{B}(H)$. Then $T$ is $\ast$-paranormal if and only if
		\begin{align*}
			{T^*}^2T^2-2kTT^*+k^2I\geq 0\text{ for all }k>0.
		\end{align*}
	\end{theorem}
	\begin{lemma}\cite[Lemma 2]{TANA}
		If $T\in\mathcal{B}(H)$ is a $\ast$-paranormal operator and $M\subseteq H$ be a closed subspace invariant under $T$, then $T|_{M}$ is $\ast$-paranormal.
	\end{lemma}
	Let $T\in\mathcal{B}(H_1,H_2)$. Then $T$ is called an \textit{isometry}, if $\|Tx\|=\|x\|$ for every $x\in H_1$, and $T$ is called a \textit{partial isometry}, if $T$ is an isometry on $N(T)^{\perp}$. % In this case, $N(T)^{\perp}$ is called the \textit{initial space} and $\overline{R(T)}$ is called the \textit{final space} of $T$.
	For every operator $T$ there exist a unique partial isometry $V\in\mathcal{B}(H_1,H_2)$ such that $T=V|T|$, where $|T|=(T^*T)^{1/2}$ is called the \textit{modulus} of $T$. This factorization is called the \textit{polar decomposition} of $T$.
	
	Let $T\in\mathcal{B}(H_1,H_2)$. Then $T$ is called a \textit{norm attaining} operator, if $\|Tx\|=\|T\|\|x\|$ for some non-zero $x\in H_1$. We say $T$ is \textit{absolutely norm attaining}, if for every closed subspace $M\subseteq H_1$, $T|_{M}$ is norm attaining. We denote the set of all absolutely norm attaining operators in $\mathcal{B}(H_1,H_2)$ by $\mathcal{AN}(H_1,H_2)$. For more details about $\mathcal{AN}$-operators, we refer to \cite{HYPO,CARVAJAL,RAM,VENKU}.
	\begin{theorem}\cite[Corollaries 2.10, 2.11]{VENKU}
		Let $T\in\mathcal{B}(H_1,H_2)$. Then the following are equivalent.
		\begin{enumerate}
			\item $T\in\mathcal{AN}(H_1,H_2)$,
			\item $T^*T\in\mathcal{AN}(H_1)$,
			\item  $|T|\in\mathcal{AN}(H_1)$.
		\end{enumerate}
	\end{theorem}
For $T\in\mathcal{B}(H)$, the \textit{minimum modulus} of $T$, $m(T)=\inf\{\|Tx\|:x\in H,\|x\|=1\}$ and \textit{ essential minimum modulus} of $T$,
$m_e(T)=\{\lambda:\lambda\in\sigma_{\text{ess}}(|T|)\}$.
	\begin{theorem}\label{thm ess spec positive AN}\cite[Theorem 2.4]{RAM}\label{thm positive AN op}
		Let $T\in\mathcal{B}(H)$ be a positive operator. Then $T\in\mathcal{AN}(H)$ if and only if $\sigma_{\text{ess}}(T)$ is singleton and $[m(T),m_e(T))$ contains at most finitely many points of $\sigma(T)$.
	\end{theorem}

	\section{Main results}

	In this section we study the structure of $\ast$-paranormal $\mathcal{AN}$-operators. We can compare these results with those of \cite{HYPO}.
	Let $T\in \mathcal B(H)$. Define
	\begin{align*}
		\mathcal{M}_{*}:=\{x\in H:\|T^*x\|=\|T\|\|x\|\},\,\mathcal{M}=\{x\in H:\|Tx\|=\|T\|\|x\|\}.
	\end{align*}
	\begin{lemma}
		Let $T\in \mathcal{B}(H)$ be a norm attaining $\ast$-paranormal. Then $\mathcal{M}_{*}\ne 0$ and $ \mathcal M_{*}\subseteq \mathcal{M}$. In fact, $\mathcal M_{*}=N(|T|-\|T\|I)\cap N(|T^*|-\|T\|I)$.
	\end{lemma}
	Note that $\mathcal{M}_{*}=N(TT^*-\|T\|^2I)=N(|T^*|-\|T\|I)$ and $\mathcal{M}=N(T^*T-\|T\|^2I)=N(|T|-\|T\|I)$.
	For $x\in\mathcal{M}_{*}$, we have
	\begin{align*}
		\|T\|^2\|x\|^2=\|T^*x\|^2\leq\|T^2x\|\|x\|\leq\|T\|\|Tx\|\|x\|\leq\|T\|^2\|x\|^2.
	\end{align*}
	From this inequality, we get that $\|Tx\|=\|T\|\|x\|$ for every $x\in\mathcal{M}_{*}$. Consequently $\mathcal{M}_{*}\subseteq\mathcal{M}$ and $\mathcal{M}_{*}=\{x\in H: \|T^*x\|=\|T\|\|x\|=\|Tx\|\}=N(TT^*-\|T\|^2I)\cap N(T^*T-\|T\|^2I)=N(|T|-\|T\|I)\cap N(|T^*|-\|T\|I)$.
	\begin{remark}
		If $T\ne 0$ is norm attaining and $\ast$-paranormal, then $\mathcal{M}_{*}\subseteq N(|T|)^{\perp}\cap N(|T^*|)^{\perp}=N(T)^{\perp}\cap N(T^*)^{\perp}$.
	\end{remark}
	
	From the following result, we get that $\mathcal{M}_{*}$ is invariant under norm attaining $\ast$-paranormal operators and is reducing if it is finite-dimensional.
	\begin{lemma}\label{lemma invariant subspace}
		Let $T\in\mathcal{B}(H)$ be a norm attaining $\ast$-paranormal operator. Then $\mathcal{M}_{*}$ is an invariant subspace for $T$. Further, if $\mathcal{M}_{*}$ is finite-dimensional, then $\mathcal{M}_{*}$ is a reducing subspace for $T$.
	\end{lemma}
	\begin{proof}
		For $x\in\mathcal{M}_{*}$, we have
		\begin{align*}
			\|T^*Tx\|=\|T\|^2\|x\|=\|T\|\|Tx\|.
		\end{align*}
		Thus $Tx\in\mathcal{M}_{*}$ for every $x\in\mathcal{M}_{*}$. Hence $\mathcal{M}_{*}$ is an invariant subspace for $T$.
		
		Now, we assume that $\mathcal{M}_{*}$ is finite-dimensional. By the definition of $\mathcal{M}_{*}$, it is easy to observe that $\frac{T|_{\mathcal{M}_{*}}}{\|T\|}$ is an isometry. Since $\mathcal{M}_{*}$ is finite-dimensional, we have $\frac{T|_{\mathcal{M}_{*}}}{\|T\|}$ is a unitary operator. The operator $T$ can be represented as
		\[T=\begin{bmatrix}
			T|_{\mathcal{M}_{*}}&A\\
			0&B
		\end{bmatrix},\]
		where $A\in\mathcal{B}(\mathcal{M}_{*}^{\perp},\mathcal{M}_{*})$ and $B\in\mathcal{B}(\mathcal{M}_{*}^{\perp},\mathcal{M}_{*}^{\perp})$. Then
		\[TT^*=
		\begin{bmatrix}
			\|T\|^2I+AA^*&AB^*\\
			BA^*&BB^*
		\end{bmatrix}\]
		and \[{T^*}^2T^2=\begin{bmatrix}
			\|T\|^4I&\|T\|^2(T|_{\mathcal{M}_{*}})^*A+{(T|_{\mathcal{M}_{*}})^*}^2AB\\
			\|T\|^2A^*(T|_{\mathcal{M}_{*}})+B^*A^*(T|_{\mathcal{M}_{*}})^2&\|T\|^2A^*A+A^*(T|_{\mathcal{M}_{*}})^2AB\\
			&+B^*A^*(T|_{\mathcal{M}_{*}})A+B^*A^*AB+{B^*}^2B^2
		\end{bmatrix}.\]
		Using the above matrix representations of $TT^*$, ${T^*}^2T^2$ and the fact that ${T^*}^2T^2-2\|T\|^2TT^*+\|T\|^4I\geq 0$, we get the $(1,1)$ entry of ${T^*}^2T^2-2\|T\|^2TT^*+\|T\|^4I$ is
		\begin{align*}
			\|T\|^4I-2\|T\|^2\left(\|T\|^2I+AA^*\right)+\|T\|^4,
		\end{align*}
		which is positive.
		This implies $-AA^*\geq 0.$ As a consequence we get that $A=0$ and
		$T=\begin{bmatrix}
			T|_{\mathcal{M}_{*}}&0\\
			0&B
		\end{bmatrix}.$
		Hence $\mathcal{M}_{*}$ is a reducing subspace for $T$.
	\end{proof}
	\begin{corollary}\label{remark}
			Let $T\in\mathcal{AN}(H)$ be a $\ast$-paranormal operator. If $\mathcal{M}_{*}$ is finite-dimensional, then $\mathcal{M}_{*}=\mathcal{M}$.
		\end{corollary}
		\begin{proof}
			Let $\mathcal{M}_{*}$ be finite dimensional. Then applying Lemma \ref{lemma invariant subspace} we get that $\mathcal{M}_{*}$ is a reducing subspace for $T$ and hence we have
			\begin{align*}
				T=\begin{bmatrix}
					T|_{\mathcal{M}_{*}}&0\\
					0&B
				\end{bmatrix},\,T^*=\begin{bmatrix}
					T^*|_{\mathcal{M}_{*}}&0\\
					0&B^*
				\end{bmatrix}.
			\end{align*}
			We claim that $\|B\|<\|T\|$. Since $T\in\mathcal{AN}(H)$, we have $T|_{\mathcal{M}_{*}^{\perp}}$ is norm attaining. Thus $B^*=T^*|_{\mathcal{M}_{*}^{\perp}}=\left(T|_{\mathcal{M}_{*}^{\perp}}\right)^*$ is norm attaining by \cite{CARVAJAL} and $\|B\|^2\in\sigma_p(BB^*)$. By definition of $\mathcal{M}_{*}$, we have that $\|B^*\|<\|T^*\|$ and consequently $\|B\|<\|T\|$.
			
			We know that $\mathcal{M}_{*}\subseteq\mathcal{M}$. Let $x\in\mathcal{M}$. Then $x=x_1+x_2$ for $x_1\in\mathcal{M}_{*}$ and $x_2\in\mathcal{M}_{*}^{\perp}\cap \mathcal M$. Then
			\begin{equation*}
				\|T\|^2x=\|T\|^2(x_1+x_2)=T^*Tx=T^*T(x_1+x_2)=\|T\|^2x_1+B^*Bx_2.
			\end{equation*}
			From the above equality, we have $\|T\|^2x_2=B^*Bx_2$, but $\|B\|<\|T\|$, which implies $x_2=0$ and $x=x_1\in\mathcal{M}_{*}$. Hence $\mathcal{M}_{*}=\mathcal{M}$.
	\end{proof}
In the followings results, we discuss a representation of absolutely norm attaining $\ast$-paranormal operators. This representation depends on the spectrum of $|T|$, which is described in Theorem \ref{thm star paranormal}. Hence based on this, we discuss several cases. 

Recall that, if $T\in\mathcal{AN}(H)$, then $|T|\in\mathcal{AN}(H)$. Then 
$$\sigma(|T|)\subseteq[m(T),m_e(T)]\cup(m_e(T),\|T\|],$$
by Theorem \ref{thm ess spec positive AN}. Also, note that $[m(T),m_e(T))\cap\sigma(|T|)$ is a finite set. Here $m_e(T)$ can be the only spectral point in $\sigma_{\text{ess}}(|T|)$. There are only two possibilities for $m_e(T)$, namely it can be eigenvalue with infinite multiplicity or a limit point of the spectral values in $(m_e(T),\|T\|]$, or both. The following result about positive $\mathcal{AN}$ operators is frequently used in the subsequent part of the article.
\begin{theorem}\cite[Theorem 3.8]{PANDEY}
	Let $T\in\mathcal{AN}(H)$ be a positive operator. Then 
	$$T=\underset{\alpha\in\Lambda}{\sum}\beta_{\alpha}v_{\alpha}\otimes v_{\alpha},$$
	where $\{v_{\alpha}:\alpha\in\Lambda\}$ is an orthonormal basis for $H$ consisting entirely of eigenvectors of $T$ and for every $\alpha\in\Lambda$, $Tv_{\alpha}=\beta_{\alpha}v_{\alpha}$ with $\beta_{\alpha}\geq 0$ such that:
	\begin{enumerate}
		\item  for every non-empty set $\Gamma\subseteq\Lambda$, we have $\sup\{\beta_{\alpha}:\alpha\in\Gamma\}=\max\{\beta_{\alpha}:\alpha\in\Gamma\}$,
		\item the spectrum $\sigma(T)=\overline{\{\beta_{\alpha} : \alpha\in\Lambda\}}$ has at most one limit point. Moreover, this unique limit point (if exists) can only be the limit of an increasing sequence in the spectrum,
		\item the set $\{\beta_{\alpha}: \alpha\in\Lambda\}$ of eigenvalues of $T$, without counting multiplicities, is countable and has at most one eigenvalue with infinite multiplicity,
		\item  if $\sigma(T)$ has both a limit point and an eigenvalue with infinite multiplicity, then they must be the same.
	\end{enumerate}
	Here $v_{\alpha}\otimes v_{\alpha}(x)=\langle x,v_{\alpha}\rangle v_{\alpha}$ for $x\in H$ and $\alpha\in \Lambda$.
\end{theorem}
	\begin{proposition}\label{prop norm as essential spec}
		Let $T\in\mathcal{AN}(H)$ be a $\ast$-paranormal operator with $\sigma_{\text{ess}}(|T|)=\{\|T\|\}$. Then
		\[T=\begin{bmatrix}
			\|T\|S_0&A\\
			0&B
		\end{bmatrix},\]
		and $S_0^*A=0$, where $S_0\in\mathcal{B}(\mathcal{M}_{*})$ is an isometry and $A\in\mathcal{B}(\mathcal{M}_{*}^{\perp},\mathcal{M}_{*})$, $B\in\mathcal{B}(\mathcal{M}_{*}^{\perp},\mathcal{M}_{*}^{\perp}).$
	\end{proposition}
	\begin{proof}
		Let $T=V|T|$ be the polar decomposition of $T$. For $x\in\mathcal{M}_{*}$, we have
		$$Tx=V|T|x=\|T\|Vx.$$
		This implies $T|_{\mathcal{M}_{*}}=\|T\|V|_{\mathcal{M}_{*}}$, but $V|_{\mathcal{M}_*}$ is an isomtery, thus $T|_{\mathcal{M}_{*}}=\|T\|S_0$, where $S_0=V|_{\mathcal{M}_{*}}$.
		
		Since $\mathcal{M}_{*}$ is invariant under $T$, we have
		\[T=\begin{bmatrix}
			\|T\|S_0&A\\
			0&B
		\end{bmatrix},\]
		where $A=P_1TP_2|_{\mathcal M_{*}^{\bot}}$ and $B=P_2TP_2|_{\mathcal M_{*}^{\bot}}$, $P_1$ is orthogonal projection onto $\mathcal{M}_{*}$ and $P_2=I-P_1$.
		Then \[TT^*=\begin{bmatrix}
			\|T\|^2S_0S_0^*+AA^*&AB^*\\
			BA^*&BB^*
		\end{bmatrix},\]
		and
		\[{T^*}^2T^2=\begin{bmatrix}
			\|T\|^4I&\|T\|S_0^*A+\|T\|^3{S_0^*}^2AB\\
			\|T\|^3A^*S_0+\|T\|^2B^*A^*S_0^2&A^*A+\|T\|A^*S_0^*AB+\|T\|B^*A^*S_0A\\
			&+B^*A^*AB+{B^*}^2B^2
		\end{bmatrix}.\]
		The $(1,1)$ entry of the $2\times 2$ operator  matrix of  ${T^*}^2T^2-2\|T\|^2TT^*+\|T\|^4I$ is
		$$2\|T\|^4-2\|T\|^4S_0S_0^*-2\|T\|^2AA^*.$$
		Since $T$ is $\ast$-paranormal, we have ${T^*}^2T^2-2\|T\|^2TT^*+\|T\|^4I\geq 0$. As a result
		$$2\|T\|^4-2\|T\|^4S_0S_0^*-2\|T\|^2AA^*\geq 0.$$
		We know that $S_0$ is isometry, thus the above equation is equivalent to
		\begin{align*}
			\|T\|^4S_0^*S_0\geq&\|T\|^4S_0^*S_0S_0^*S_0+\|T\|^2S_0^*AA^*S_0\\
			=&\|T\|^4+\|T\|^2S_0^*AA^*S_0.
		\end{align*}
		Thus $S_0^*AA^*S_0\leq 0$ and consequently $S_0^*A=0$.
	\end{proof}
	\begin{lemma}\label{lemma2}
		Let $T\in\mathcal{AN}(H)$ be a $\ast$-paranormal operator and $[m_e(T),\|T\|]\cap 
		\sigma(|T|)$ is a finite set of cardinality $n_1\in\mathbb{N}$. Then $T$ has the following representation:
		\[T=\begin{bmatrix}
			\underset{i=1}{\overset{n_1}{\oplus}}\lambda_i S_i&0&0\\
			0&m_e(T)S&A\\
			0&0&B
		\end{bmatrix},\]
		with respect to the decomposition $H=H_1\oplus H_2\oplus H_3$, where
		\begin{enumerate}
			\item $\lambda_i$ are positive real number for every $1\leq i\leq n_1$.
			\item $H_1,\,H_3$ are finite-dimensional subspaces of $H$.
			\item $S_i\in\mathcal{B}(N(|T^*|-\lambda_iI))$ are unitary operators, for $1\leq i\leq n_1$.
			\item $S\in\mathcal{B}(N(|T^*|-m_e(T)I))$ is an isometry.
			\item $S^*A=0$.
		\end{enumerate}
	\end{lemma}
	\begin{proof}
		Let $T\in\mathcal{AN}(H)$. Then $|T|\in\mathcal{AN}(H)$ and $\sigma_{\text{ess}}(|T|)$ is singleton, which is equal to $\{m_e(T)\}$. 
		
		Case $(1):$ $m_e(T)=\|T\|$. Then the result follows by Proposition \ref{prop norm as essential spec}. In this case $H_1=\{0\}$, $H_2=N(|T^*|-\|T\|I)$ and $H_3=H_2^{\perp}$.
		
		Case $(2):$ $m_e(T)<\|T\|$. By given conditions, $(m_e(T),\|T\|]\cap \sigma(|T|)=\{\lambda_i:1\leq i\leq n_1\}$, where $\lambda_i<\lambda_j$ for $j<i$. Note that $\lambda_1=\|T\|$ and by Lemma \ref{lemma invariant subspace} $G_1=N(|T|-\lambda_1I)=N(|T^*|-\lambda_1I)$ is a reducing subspace for $T$ and $T|_{G_1}=\lambda_1S_1$, where $S_1\in\mathcal{B}(G_1)$ is a unitary operator. Then $T|_{G_1^{\perp}}\in\mathcal{AN}$ is $\ast$-paranormal and $\lambda_2=\|T|_{G_1^{\perp}}\|$. Again by Lemma \ref{lemma invariant subspace}, $G_2=N(|T|-\lambda_2I)=N(|T^*|-\lambda_2 I)$ is a reducing subspace for $T$ and $T|_{G_2}=\lambda_2 S_2$, where $S_2$ is a unitary operator. Continuing in this way, we get $T|_{H_1}=\underset{i=1}{\overset{n_1}{\oplus}}\lambda_i S_i$, where $H_1=\underset{i=1}{\overset{n_1}{\oplus}}N(|T|-\lambda_iI)=\underset{i=1}{\overset{n_1}{\oplus}}N(|T^*|-\lambda_iI)$.
		
		Since $(m_e(T),\|T\|]\cap\sigma(|T|)$ is finite set, we have $m_e(T)$ is an eigenvalue of $|T|$ with infinite multiplicity. Also, note that $\|T|_{H_1^{\perp}}\|=m_e(T)$ and $\{m_e(T)\}=\sigma_{\text{ess}}\left(\left|T|_{H_1^{\perp}}\right|\right)=\sigma_{\text{ess}}(|T||_{H_1^{\perp}})$. By Proposition \ref{prop norm as essential spec}, we have 
		\[T|_{H_1^{\perp}}=\begin{bmatrix}
			m_e(T)S&A\\
			0&B
		\end{bmatrix},\]
		where $H_1^{\perp}=H_2\oplus H_3$, $H_2=N(|T^*|-m_e(T)I)$ and $H_3=(H_1\oplus H_2)^{\perp}$. Hence
		\[T=\begin{bmatrix}
			\underset{i=1}{\overset{n_1}{\oplus}}\lambda_i S_i&0&0\\
			0&		m_e(T)S&A\\
			0&	0&B
		\end{bmatrix}.\]
	This proves the result.
	\end{proof}
	\begin{lemma}\label{lemma3}
		Let $T\in\mathcal{AN}(H)$ be a $\ast$-paranormal operator and $(m_e(T),\|T\|]\cap\sigma(|T|)$ is an infinite set. Then $T$ has the following representation:
		\[T=\begin{bmatrix}
			\underset{i=1}{\overset{\infty}{\oplus}}\lambda_i S_i&0&0\\
			0&m_e(T)S&A\\
			0&0&B
		\end{bmatrix},\]
		with respect to the decomposition $H=H_1\oplus H_2\oplus H_3$, where
		\begin{enumerate}
			\item $\lambda_i$ are positive real number for every $1\leq i<\infty$.
			\item $H_3$ is finite-dimension subspaces of $H$.
			\item $S_i\in\mathcal{B}(N(|T^*|-\lambda_iI))$ are unitary operators, for $1\leq i<\infty$.
			\item $S\in\mathcal{B}(N(|T^*|-m_e(T)))$ is an isometry.
			\item $S^*A=0$.
		\end{enumerate}
	\end{lemma}
	\begin{proof}
		As in Case $(2)$ of Lemma \ref{lemma2}, we get $T|_{H_1}=\underset{i=1}{\overset{\infty}{\oplus}}\lambda_iS_i$, where $H_1=\underset{i=1}{\overset{\infty}{\oplus}} N(|T|-\lambda_iI)=\underset{i=1}{\overset{\infty}{\oplus}}N(|T^*|-\lambda_iI)$ and $S_i\in\mathcal{B}(N(|T|-\lambda_iI))$ is a unitary operator for $1\leq i<\infty$. 
		
		Case $(1):$ $m_e(T)$ is a limit point of $\{\lambda_j\}$, but $m_e(T)$ is not an eigenvalue of $|T|$. 
		
		Subcase $(a):$ $m(T)=m_e(T)$. Then $H_1=\underset{i=1}{\overset{\infty}{\oplus}}N(|T|-\lambda_iI)=H$, $H_2=\{0\}=H_3$ and $T=\underset{i=1}{\overset{\infty}{\oplus}}\lambda_i S_i$.
		
		Subcase $(b):$ $m(T)<m_e(T)$. Then $[m(T),m_e(T))\cap\sigma(|T|)=\{\delta_i:1\leq i\leq n_2\}$ for $n_2\in\mathbb{N}$, where $\delta_i<\delta _j$ for $j<i$. We have $\|T|_{H_1^{\perp}}\|=\delta_1<m_e(T)$. Also $T|_{H_1^{\perp}}\in\mathcal{AN}$ is $\ast$-paranormal. By Lemma \ref{lemma invariant subspace}, $D_1=N(|T|-\delta_1I)=N(|T^*|-\delta_1I)$ and $T|_{D_1}=\delta_1U_1$, where $U_1\in\mathcal{B}(D_1)$ is a unitary operator. Now $T|_{(H_1\oplus D_1)^{\perp}}\in\mathcal{AN}$ is $\ast$-paranormal and $\|T|_{(H_1\oplus D_1)^{\perp}}\|=\delta_2$. Again using Lemma \ref{lemma invariant subspace}, $D_2=N(|T|-\delta_2I)=N(|T^*|-\delta_2I)$ and $T|_{D_2}=\delta_2 U_2$, where $U_2$ is a unitary operator. Continuing this way, we get $T|_{H_2}=\underset{i=1}{\overset{n_2}{\oplus}}\delta_iU_i$, where $H_2=\underset{i=1}{\overset{n_2}{\oplus}}N(|T|-\delta_iI)$. In this case, $H_2=\{0\}$, and \[T=\begin{bmatrix}
			\underset{i=1}{\overset{\infty}{\oplus}}\lambda_iS_i&0\\
			0&\underset{i=1}{\overset{n_2}{\oplus}}\delta_i U_i
		\end{bmatrix}.\]
		
		Case $(2):$ $m_e(T)$ is a limit point of $\{\lambda_i\}$ as well an eigenvalue of $|T|$. Note that $\|T|_{H_1^{\perp}}\|=m_e(T)$.
		
		Subcase $(a):$ $m_e(T)$ is an eigenvalue of finite multiplicity and $m(T)=m_e(T)$. By Lemma \ref{lemma invariant subspace}, $N(|T|-m_e(T)I)=N(|T^*|-m_e(T)I)$ is a reducing subspace for $T$ and $T|_{N(|T|-m_e(T)I)}=m_e(T)S$, where $S\in \mathcal{B}(N(|T|-m_e(T)I))$ is a unitary operator. Hence $H_2=N(|T|-m_e(T)I)$, $H_3=\{0\}$ and \[T=\begin{bmatrix}
			\underset{i=1}{\overset{\infty}{\oplus}}\lambda_iS_i&0\\
			0&m_e(T)S
		\end{bmatrix}.\]
		
		Subcase $(b):$ $m_e(T)$ is an eigenvalue of finite multiplicity and $m(T)<m_e(T)$. By Subacse $(b)$ of Case $(2)$ and following similar steps as in Subcase $(b)$ of Case $(1)$, we get
		\[T=
		\begin{bmatrix}
			\underset{i=1}{\overset{\infty}{\oplus}}\lambda_iS_i&0&0\\
			0&m_e(T)S&0\\
			0&0&\underset{i=1}{\overset{n_2}{\oplus}}\delta_i U_i
		\end{bmatrix},
		\] where $H_1=\underset{i=1}{\overset{\infty}{\oplus}}N(|T|-\lambda_iI)$, $H_2=N(|T|-m_e(T)I)$ and $H_3=\underset{i=1}{\overset{n_2}{\oplus}}N(|T|-\delta_iI)$.
		
		Subcase $(c):$ $m_e(T)$ is an eigenvalue of infinite multiplicity. In This case $\|T|_{H_1^{\perp}}\|=m_e(T)$ and $\{m_e(T)\}=\sigma_{\text{ess}}(|T|_{H_1^{\perp}}|)=\sigma_{\text{ess}}(|T||_{H_1^{\perp}})$. By Proposition \ref{prop norm as essential spec}, we have
		\[T|_{H_1^{\perp}}=\begin{bmatrix}
			m_e(T)S&A\\
			0&B
		\end{bmatrix},\] where $H_2=N(|T^*|-m_e(T)I)$ and $H_3=(H_1\oplus H_2)^{\perp}$.
	\end{proof}
	
	%\begin{lemma}\label{lemma4}
	%Let $T\in\mathcal{AN}(H)$ be a $\ast$-paranormal operator and $m_e(T)=m(T)$. Then $T$ has the following representation:
	%\[
	%	T=
	%	\begin{bmatrix}
		%		\underset{x\in\Lambda_1}{\oplus}\lambda_xS_x&0\\
		%		0&m_e(T) S
		%	\end{bmatrix},
	%\]
	%with respect to the decomposition $H=H_1\oplus H_2$, where
	%\begin{enumerate}
	%	\item $(m_e(T),\|T\|]\cap\sigma(|T|)=\{\lambda_x:x\in\Lambda_1\subseteq\mathbb{N}\cup\emptyset\}$,
	%	\item $H_1=\underset{x\in\Lambda_1}{\oplus}N(|T^*|-\lambda_xI)$, $H_2=N(|T^*|-m_e(T) I)$,
	%	\item $S_x\in\mathcal{B}(N(|T^*|-\lambda_x I))$ is unitary operator for $x\in\Lambda_1$ and $S\in\mathcal{B}(N(|T^*|-m_e(T) I))$ is an isometry.
	%\end{enumerate}
	%\end{lemma}
	%\begin{proof}
	%	Following the similar arguments, as in Lemma \ref{lemma2} and \ref{lemma3}, we get $T|_{H_1}=\underset{x\in\Lambda_1}{\oplus}\lambda_xS_x$, where $S_x\in\mathcal{B}(N(|T^*|-\lambda_x I))$ is unitary operator. 
	
	%	If $m_e(T)\notin \sigma_p(|T|)$, then $H_2=\{0\}$ and $T=\underset{x\in\Lambda_1}{\oplus}\lambda_xS_x$. If $m_e(T)\in\sigma_p(|T|)$, then $\|T|_{H_1^{\perp}}\|=m_e(T) $ and $\sigma\left(\left|T|_{H_1^{\perp}}\right|\right)=\{m_e(T)\}$. If $m_e(T)=0$, then $H_2=N(T)$ and the result follows by taking $S=I$. If $m_e(T)\ne 0$, then $\left|T|_{H_1^{\perp}}\right|=m_e(T)I$ and by Polar decomposition, $T_{H_1^{\perp}}=m_e(T)S$ for some isometry $S$.
	%\end{proof}
	We summarize the above results in the following theorem.
	\begin{theorem}\label{thm star paranormal}
		Let $T\in\mathcal{AN}(H)$ be a $\ast$-paranormal operator. Then $T$ can be represented as
		\begin{equation}\label{eqn matrix rep}
			T=
			\begin{bmatrix}
				\underset{x\in\Lambda}{\oplus}\lambda_xS_x&0&0\\
				0&\lambda S&A\\
				0&0&B
			\end{bmatrix},
		\end{equation}
		with respect to the decomposition $H=H_1\oplus H_2\oplus H_3$, where
		\begin{enumerate}
				\item  $\Lambda=(m_e(T),\|T\|]\cap\sigma(|T|)$,
			\item $\sigma_{\text{ess}}(|T|)=\{\lambda\}$,
			\item $(\lambda,\|T\|]\cap\sigma(|T|)=\{\lambda_x:x\in\Lambda\subseteq\mathbb{N}\cup\{\infty\}\cup\emptyset\}$,
			\item $H_1=\underset{x\in\Lambda}{\oplus}N(|T^*|-\lambda_xI)$, $H_2=N(|T^*|-\lambda I)$ and $H_3=(H_1\oplus H_2)^{\perp}$ is finite-dimensional,
			\item $S_x\in\mathcal{B}(N(|T^*|-\lambda_x I))$ is unitary operator for $x\in\Lambda$ and $S\in\mathcal{B}(N(|T^*|-\lambda I))$ is an isometry,
			\item $S^*A=0$.
		\end{enumerate}
	\end{theorem}
	\begin{proof}
		The proof follows by Lemma \ref{lemma2}, and \ref{lemma3}.
	\end{proof}
\begin{theorem}
	Let $T\in\mathcal{AN}(H)$ be a $\ast$-paranormal operator. Then $T=U\oplus D$, where $U$ is a direct sum of scalar multiple of unitary operators and $D$ is a $2\times 2$ upper diagonal matrix.
\end{theorem}
\begin{proof}
	The proof follows from Theorem \ref{thm star paranormal}, by taking $H=H_1\oplus (H_2\oplus H_3)$, $U=\underset{x\in\Lambda}{\oplus}\lambda_xS_x$ and 
	\[D=\begin{bmatrix}
		m_e(T)S&A\\
		0&B
	\end{bmatrix},\]
 where $H_i$ for $i=1,2,3$, $S$, $A$ and $B$ are as defined in Theorem \ref{thm star paranormal}.
\end{proof}
	From \cite{HYPO}, we know that a hyponormal absolutely norm attaining operator also has a similar representation as (\ref{eqn matrix rep}). Here, we want to imphasize that any operator of the form (\ref{eqn matrix rep}) need not be hyponormal.
	\begin{example}
		Let $T:\ell^2(\mathbb{N})\oplus\mathbb{C}^2\rightarrow\ell^2(\mathbb{N})\oplus\mathbb{C}^2$ be a bounded linear operator defined by
		\begin{equation*}
			T\left((x_1,x_2,\ldots),(y_1,y_2)\right)=\left((y_1,2x_1,2x_2,\ldots),(y_1,y_2)\right),\; \text{for all}\; (x_n)\ell^2.
		\end{equation*}
		Then
		\begin{align*}
			T^*T((x_1,x_2,\ldots),(y_1,y_2))&=\left((4x_1,4x_2,\ldots),(2y_1,y_2)\right),\\
			T^*T((x_1,x_2,\ldots),(y_1,y_2))&=\left((y_1+x_1,4x_2,\ldots),(x_1+y_1,y_2)\right).
		\end{align*}
		Cleary $T$ is not hyponormal. But $T$ can be written as
		\[T=\begin{bmatrix}
			2S&A\\
			0&B
		\end{bmatrix},\]
		where $S:\ell^2(\mathbb{N})\rightarrow\ell^2(\mathbb{N})$ is the right shift operator
		\begin{align*}
			S(x_1,x_2,\ldots)=(0,x_1,x_2,\ldots),
		\end{align*}
		$A:\mathbb{C}\oplus\mathbb{C}\rightarrow\ell^2(\mathbb{N})$ and $B:\mathbb{C}\oplus\mathbb{C}\rightarrow\mathbb{C}\oplus\mathbb{C}$ are defined by
		\begin{align*}
			A(x_1,x_2)=(x_1,0,\ldots),\,B(x_1,x_2)=(x_1,x_2),\; \text{for all}\; (x_1,x_2)\in \mathbb C\oplus \mathbb C.
		\end{align*}
		Also, $S^*A=0$.
	\end{example}
The following result is of independent interest, which is used in the subsequent results as well. This is analogous to \cite[Problem 72]{HALMOS} in which dim $H_1<\infty$ is assumed, while in our case we assume that dim $H_2<\infty$. 
	\begin{lemma}\label{lemma1}
		Let $T:H_1\oplus H_2\rightarrow H_1\oplus H_2$ be a bounded linear operator and $\dim H_2<\infty$.
		If \[T=\begin{bmatrix}
			A&B\\
			0&C
		\end{bmatrix}\] is invertible, then $A$ and $C$ are invertible and
		\[
		\begin{bmatrix}
			A&B\\
			0&C
		\end{bmatrix}^{-1}=\begin{bmatrix}
			A^{-1}&-A^{-1}BC^{-1}\\
			0&C^{-1}
		\end{bmatrix}.\]
	\end{lemma}
	\begin{proof}
		If
		$\begin{bmatrix}
			X&Y\\
			Z&W
		\end{bmatrix}$
		is the inverse of 	$\begin{bmatrix}
			A&B\\
			0&C
		\end{bmatrix}$, then
		\[\begin{bmatrix}
			A&B\\
			0&C
		\end{bmatrix}\begin{bmatrix}
			X&Y\\
			Z&W
		\end{bmatrix}=\begin{bmatrix}
			I_{H_1}&0\\
			0&I_{H_2}
		\end{bmatrix}.\]
		Then
		\begin{align}
			AX+BZ=&I_{H_1},\label{1}\\
			AY+BW=&0,\label{2}\\
			CZ=&0,\label{3}\\
			CW=I_{H_2}.\label{4}
		\end{align}
		From (\ref{4}) $W$ is the right inverse of $C$. As dim$H_2<\infty$, we have $W=C^{-1}$. Now from (\ref{3}), we have $Z=0$. Hence by $(\ref{1})$, we have $AX=I$. That $X$ is the right inverse of $A$. Now from $(\ref{2})$, $AYC^{-1}+B=0$ or $AYC^{-1}=-B$. Also we have
		\[\begin{bmatrix}
			X&Y\\
			Z&W
		\end{bmatrix}\begin{bmatrix}
			A&B\\
			0&C
		\end{bmatrix}=\begin{bmatrix}
			I_{H_1}&0\\
			0&I_{H_2}
		\end{bmatrix}.\]
		We get
		\begin{align}
			XA=&I_{H_1},\label{5}\\
			XB+YC=&0,\label{6}\\
			ZA=&0,\label{7}\\
			ZB+WC=&I_{H_2}.\label{8}
		\end{align}
		From (\ref{5}), $X$ is left inverse of $A$. Thus $X$ is the inverse of $A$. By (\ref{2}), we have $AY=-BW$ or $Y=-A^{-1}BC^{-1}.$ Thus
		\[
		\begin{bmatrix}
			X&Y\\
			Z&W
		\end{bmatrix}=
		\begin{bmatrix}
			A^{-1}&-A^{-1}BC^{-1}\\
			0&C^{-1}
		\end{bmatrix}.\]
	\end{proof}
	\begin{corollary}\label{Coro diagonal matrix}
		Let $T$ be as defined in Lemma \ref{lemma1} with $A=\alpha S$, where $S$ is an isometry, $\alpha\in\mathbb{R}\setminus\{0\}$. If $S^*B=0$, then $B=0$ and
		\[T=\begin{bmatrix}
			A&0\\
			0&C
		\end{bmatrix}.\]
	\end{corollary}
	\begin{proof}
		By Lemma \ref{lemma1}, we can conclude that $A$ and $C$ are invertible. This imply that $S$ is unitary. Also, we have
		\begin{align*}
			\begin{bmatrix}
				\alpha S&B\\
				0&C
			\end{bmatrix}^{-1}=&\begin{bmatrix}
				\alpha^{-1}S^{-1}&-\alpha^{-1}S^{-1}BC^{-1}\\
				0&C^{-1}
			\end{bmatrix}\\
			=&\begin{bmatrix}
				\alpha^{-1}S^*&-\alpha^{-1}S^*BC^{-1}\\
				0&C^{-1}
			\end{bmatrix}\\
			=&\begin{bmatrix}
				\alpha^{-1}S^*&0\\
				0&C^{-1}
			\end{bmatrix}
		\end{align*}
		Therefore we have
		\[T^{-1}=\begin{bmatrix}
			\alpha^{-1}S^*&0\\
			0&C^{-1}
		\end{bmatrix}.\]
		That is
		\[T=\begin{bmatrix}
			\alpha S&0\\
			0&C
		\end{bmatrix}=\begin{bmatrix}
			A&0\\
			0&C
		\end{bmatrix}.\]
	\end{proof}
	
	Now, we give a condition under which a $\ast$-paranormal operator becomes normal.
	\begin{theorem}\label{ast paranormal is normal}
		Let $T\in\mathcal{AN}(H)$ be a $\ast$-paranormal operator. If $T$ is invertible, then $T$ is normal.
	\end{theorem}
	\begin{proof}
		As $T\in\mathcal{AN}(H)$ be a $\ast$-paranormal operator, by Theorem \ref{thm star paranormal}, $T$ can be written as
		\[T=\begin{bmatrix}
			\underset{x\in\Lambda}{\oplus}\lambda_xS_x&0&0\\
			0&\lambda S&A\\
			0&0&B
		\end{bmatrix},\]
	where $\Lambda_1$, $S_x$, $S$ $A$ and $B$ are as defined in Theorem \ref{thm star paranormal}.
		Since $T$ is invertible, we conclude that \[\begin{bmatrix}
			\lambda S&A\\
			0&B
		\end{bmatrix}\]
		is invertible.
		By Theorem \ref{thm star paranormal}, we also have $S^*A=0$ and $S$ is an isometry.
		Thus Corollary \ref{Coro diagonal matrix} implies that
		\[T=\begin{bmatrix}
			\underset{x\in\Lambda}{\oplus}\lambda_xS_x&0&0\\
			0&\lambda S&0\\
			0&0&B
		\end{bmatrix}.\]
		Hence $T$ is normal.
	\end{proof}
Here, we present another sufficient condition which gives normality of a $\ast$-paranormal operator.
	\begin{theorem}\label{thm2}
		If $\dim N(T)=\dim N(T^*)$, and $T\in\mathcal{AN}(H)$ is $\ast$-paranormal, then $T$ must be normal.
	\end{theorem}
	\begin{proof}
		By the definition of $\ast$-paranormal operators, it is clear that $N(T)\subseteq N(T^*)$. The assumption $\dim N(T)=\dim N(T^*)$ implies that $N(T)=N(T^*)$. If $T$ is compact, then by \cite[Theorem 4.6]{RASHID}, $T$ must be normal. So next assume that $T$ is non compact. In this case, by \cite[Proposition 2.8]{VENKU}, we must have that $R(T)$ is closed and $N(T)$ is finite dimensional. Since $N(T)$ is invariant under $T$ and $N(T^*)=N(T)$ is invariant under $T^*$, we conclude that $N(T)$ reduces $T$. Hence
		\[T=\begin{bmatrix}
			0&0\\
			0&\tilde{T}
		\end{bmatrix},\]
		where $\tilde{T}=T|_{N(T)^{\perp}}:N(T)^{\perp}\rightarrow N(T)^{\perp}$. Note that $\tilde{T}$ is $\ast$-paranormal, $\mathcal{AN}$-operator and invertible. Hence by Theorem \ref{ast paranormal is normal}, we have that $\tilde{T}$ is normal. Hence $T$ must be normal.
	\end{proof}
\begin{corollary}
	Let $T\in\mathcal{AN}(H)$ be a $\ast$-paranormal operator and $0\notin \sigma_w(T)$, where $\sigma_{w}(T)=\sigma_{\text{ess}}(T)\cup\{\lambda\notin\sigma_{\text{ess}}(T):\dim N(T-\lambda I)\ne\dim N((T-\lambda I)^*)\}$. Then $T$ is  normal.
\end{corollary}
\begin{proof}
	The proof follows directly from Theorem \ref{thm2}.
\end{proof}
	%\begin{corollary}\label{coro paranormal}
	%		If $N(T)=N(T^*)$, and $T\in\mathcal{AN}(H)$ is paranormal, then $T$ must be normal.
	%\end{corollary}
	%\begin{proof}
	%		First note that $T$ is hyponormal. Now, the result follows on similar lines as in Theorem \ref{thm2}.
	%\end{proof}
	The following example illustrates that, if we remove the condition $T$ is $\ast$-paranormal in Theorem \ref{thm2}, %and $T$ is paranormal in Corollary \ref{coro paranormal},
	then the results need not hold.
	\begin{example}
		Consider the operator $T:\ell^2(\mathbb{N})\oplus\ell^2(\mathbb{N})\rightarrow\ell^2(\mathbb{N})\oplus\ell^2(\mathbb{N})$ defined by\begin{align*}
			T((x_1,x_2,\ldots),(s_1,s_2,\ldots))=\left(\left(s_1,x_1,\frac{x_2}{2},\frac{x_3}{3},\ldots\right),\left(s_2,s_3,\ldots\right)\right).
		\end{align*}
		Then
		\begin{align*}
			TT^*((x_1,x_2,\ldots),(s_1,s_2,\ldots))=&\left(\left(x_1,x_2,\frac{x_3}{2^2},\frac{x_4}{3^2},\ldots\right),\left(s_1,s_2,\ldots\right)\right),\\
			T^*T((x_1,x_2,\ldots),(s_1,s_2,\ldots))=&\left(\left(x_1,\frac{x_2}{2^2},\frac{x_3}{3^2},\ldots\right),\left(s_1,s_2,\ldots\right)\right).
		\end{align*}
		It is easy to see that $T^*$ is hyponormal and $N(T)=\{0\}=N(T^*)$. But $T^*\notin\mathcal{AN}\left(\ell^2(\mathbb{N})\oplus\ell^2(\mathbb{N})\right)$ as $\sigma_{\text{ess}}(TT^*)=\{0,1\}.$
	\end{example}

	\noindent\textit{Acknowledgement:}  The research of the author is supported by  the Theoretical Statistics and Mathematics Unit, Indian Statistical Institute, Bangalore Centre, India. The author would like to thank her Ph.D. supervisor for his valuable suggestion and mathematical discussions throughout the development of the article.

\end{document}